\newcommand{\si}[1]{#1}
\newcommand{\jo}[1]{}

\newcommand{\FINANCIAMENTO}{We would like to thank Ellen H. Fukuda (Kyoto University) and Paulo J.S. Silva (University of Campinas) for initial discussions on this topic. This work was supported by
	CEPID-CeMEAI (FAPESP 2013/07375-0),
	FAPESP (grants 2018/24293-0, 2017/18308-2,
	2017/17840-2, 
	and 2017/12187-9),
	CNPq (grants
	301888/2017-5,
	303427/2018-3, and 404656/2018-8), and
	FONDECYT grant 1201982 and Basal Program CMM-AFB 170001, both from ANID (Chile).}

\si{
	\documentclass{article}
	\headheight=21.06892pt
	\addtolength{\textheight}{5cm}
	\addtolength{\topmargin}{-2.5cm}
	\setlength{\oddsidemargin}{-.4cm}
	\setlength{\textwidth}{17cm}
	\textheight=640pt}

\jo{\RequirePackage{fix-cm}
\documentclass[smallextended]{svjour3}
\smartqed
\usepackage[misc]{ifsym}
}
\usepackage{graphicx}
\usepackage{color}
\usepackage[section]{algorithm}
\usepackage[noend]{algpseudocode}
\jo{\journalname{Optimization Letters}}
\usepackage{graphicx}
\usepackage{multirow}
\usepackage{mathtools}
\usepackage{tikz} 
\usepackage{amsmath}
\usepackage[all,knot,arc,import,poly]{xy}
\usepackage{algorithm}
\usepackage[noend]{algpseudocode}
\usepackage{amsopn}
\usepackage{setspace}
\si{\usepackage{amsthm}
	\newtheorem{theorem}{Theorem}[section]

	\newtheorem{lemma}{Lemma}[section]
	\newtheorem{definition}{Definition}[section]
	\newtheorem{example}{Example}[section]
	\newtheorem{remark}{Remark}[section]
	}

	\DeclareMathOperator{\inte}{int}
	
	\DeclareMathOperator{\bdp}{bd{^+}}
	
	\newcommand{\R}{\mathbb{R}}
	\newcommand{\N}{\mathbb{N}}

	\newcommand{\half}{\frac{1}{2}}
	
	\renewcommand{\S}{\mathbb{S}}
	\newcommand{\tr}{\textnormal{trace}}
	
	\renewcommand{\Im}{\textnormal{Im}}

	\newcommand{\lin}{\textnormal{lin}}

\usepackage{amssymb,amsmath}
\usepackage{hyperref}
\hypersetup{
	colorlinks=true,
	citecolor=blue,
	linkcolor=blue,
	filecolor=magenta,      
	urlcolor=cyan,
}

\usepackage{mathptmx}  
\usepackage{latexsym}
\usepackage{mathtools}

\begin{document}

\title{Naive constant rank-type constraint qualifications for multifold second-order cone programming and semidefinite programming\jo{\thanks{\FINANCIAMENTO}}}
\date{September 11, 2020}
\jo{\titlerunning{Naive constant rank-type CQs for multifold SOCPs and SDPs}}

\si{
\author{R. Andreani\footnote{Department of Applied Mathematics, University of Campinas, Campinas-SP, Brazil. Email: andreani@ime.unicamp.br} \and G. Haeser\footnote{Department of Applied Mathematics, University of S\~ao Paulo, S\~ao Paulo-SP, Brazil. Email: \{ghaeser,leokoto,thiagops\}@ime.usp.br}\and L.M. Mito\footnotemark[2] \and H. Ram\'{\i}rez\footnote{Departamento de Ingenier\'{\i}a Matem\'atica and Centro de Modelamiento Matem\'atico (CNRS UMI 2807), Universidad de Chile, Santiago, Chile. Email: hramirez@dim.uchile.cl}\and D.O. Santos\footnote{Institute of Science and Technology, Federal University of S\~ao Paulo, S\~ao Jos\'e dos Campos-SP, Brazil. Email: daiana@ime.usp.br} \and T.P. Silveira\footnotemark[2]}

\date{September 11, 2020}}

\jo{
\author{R. Andreani \and G. Haeser \and L.M. Mito\and \mbox{H. Ram\'{\i}rez} \and D.O. Santos \and \mbox{T.P. Silveira}}
\institute{R. Andreani \at
	Department of Applied Mathematics, University of Campinas, Campinas-SP, Brazil. \email{andreani@ime.unicamp.br}
           \and
            G. Haeser \Letter\and L.M. Mito \and T.P. Silveira\at
            Department of Applied Mathematics, University of S\~ao Paulo, S\~ao Paulo-SP, Brazil. \email{\{ghaeser,leokoto,thiagops\}@ime.usp.br}
            \and
            H. Ram\'{\i}rez \at
            Departamento de Ingenier\'{\i}a Matem\'atica and Centro de Modelamiento Matem\'atico (CNRS UMI 2807), Universidad de Chile, Santiago, Chile. \email{hramirez@dim.uchile.cl}
            \and
            D.O. Santos\at
            Institute of Science and Technology, Federal University of S\~ao Paulo, S\~ao Jos\'e dos Campos-SP, Brazil.
            \email{daiana@ime.usp.br}
            }
\date{Received: date / Accepted: date}
}

\maketitle

%
%
	
\abstract{The constant rank constraint qualification, introduced by Janin in 1984 for nonlinear programming, has been extensively used for sensitivity analysis, global convergence of first- and second-order algorithms, and for computing the derivative of the value function. In this paper we discuss naive extensions of constant rank-type constraint qualifications to second-order cone programming and semidefinite programming, which are based on the Approximate-Karush-Kuhn-Tucker necessary optimality condition and on the application of the reduction approach. Our definitions are strictly weaker than Robinson's constraint qualification, and an application to the global convergence of an augmented Lagrangian algorithm is obtained.
%
\jo{
}
\si{\\[0.2cm] {\bf Keywords:} Constraint qualifications; Optimality conditions; Second-order cone programming; Semidefinite programming; Global convergence.}

%
%
\section{Introduction}

In this paper we investigate constraint qualifications (CQs) for second-order cone programming and semidefinite programming. In particular, we are interested in constant rank CQs as defined first in \cite{crcq} and later extended in \cite{rcpld,cpg,minch,cpld} in the context of nonlinear programming. In particular, the definition in \cite{crcq} gained some notoriety for its ability to compute the derivative of the value function, a result known to hold at the time only under Mangasarian-Fromovitz CQ \cite{robinson}. Also, the definition from \cite{crcq} includes naturally the case of linear constraints, which does not follow under Mangasarian-Fromovitz CQ. 
The ability to handle redundant constraints (in particular, linear ones) in the case of nonlinear programming is a powerful modeling tool that frees the model builder from the apprehension of including them without preprocessing.
Actually, the effort of finding which constraints are redundant may be equivalent to the effort of solving the problem.



%
For conic programming, it is well known that linearity of the constraints is not a CQ \cite{dualsocp,dualsdp} and this somehow stresses the difficulties in extending these ideas to the conic context. In particular, a previous tentative extension to second-order cones \cite{ZZ} has been shown to be incorrect \cite{errata}.

In this paper, we make use of the reduction approach in order to propose new constant rank-type CQs for second-order cone programming and semidefinite programming that are strictly weaker than Robinson's CQ. In our approach, we separate the constraints into two sets: one consisting of the constraints that can be completely characterized by standard equality and inequality nonlinear programming constraints, and other with the irreducible conic constraints. For second-order cone programming, the second block consists of constraints that are active at the vertex of a multi-dimensional second-order cone, while for semidefinite programming these correspond to semidefinite blocks where the zero eigenvalue is non-simple. 

We consider our conditions to be naive extensions of the corresponding nonlinear programming CQ in the sense that if the problem only has irreducible constraints then all our conditions coincide with Robinson's CQ; however we show some interesting examples where our condition holds while Robinson's CQ fails. Extending these ideas to consider also the irreducible constraints is an ongoing topic of research.

%
Despite our inability of dealing with the irreducible conic constraints, the Approximate-Karush-Kuhn-Tucker (AKKT) \cite{akkt} necessary optimality condition, recently extended to second-order cones \cite{psocp} and semidefinite programming \cite{AHV}, can easily be used to handle the remaining constraints by means of the reduction approach. This allows obtaining CQs analogous to those defined in \cite{rcpld,cpg,crcq,minch,cpld}. Analogous definitions of \cite{crcq,minch} are independent of Robinson's CQ, while analogues of \cite{rcpld,cpg,cpld} are strictly weaker than Robinson's CQ.

Since several algorithms are expected to generate AKKT sequences (this is the case, for instance, of the augmented Lagrangian algorithms of \cite{psocp} and \cite{AHV}), a relevant corollary of our analysis is that all CQs introduced in this paper can be used for proving global convergence of these algorithms to a KKT point.

This paper is organized as follows. In Section 2, we briefly introduce constant rank CQs for nonlinear programming. In Section 3, we revisit constraint qualifications for second-order cone programming. Section 4 is devoted to the AKKT approach, while in Section 5 we introduce and explain our new CQs for second-order cones. In Section 6 we extend these ideas to semidefinite programming. Finally, our conclusions are presented in Section 7.\\

{\bf Notation: }For a continuously differentiable function $g\colon\R^n\to\R^m$, we denote $J_g(x)$ the $m\times n$ Jacobian matrix of $g$ at $x$, for which the $j$-th row is given by the transposed gradient $\nabla g_j(x)^T$ of the $j$-th component function $g_j\colon\R^n\to\R, j=1,\dots,m$. Any finite-dimensional space $\R^m$ is equipped with its standard Euclidean inner product $\langle x, y \rangle:=x^T y = \sum_{j=1}^m x_j y_j$. Then, given a closed convex cone $K\subseteq\R^m$, we denote its polar by $K^\circ:=\{v\in\R^m\mid \langle v,y\rangle\leq0, \forall y\in K\}$.
Finally, we adopt the following standard conventions on the empty set $\emptyset$: the sum over  an empty index set is null (i.e., $\sum_\emptyset=0$) and  $\emptyset$  is linearly independent (considered as the basis of the trivial linear space $\{0\}$).

\section{Constant rank-type CQ conditions in nonlinear programming}

Consider the following nonlinear programming problem (NLP): 
\begin{eqnarray}\label{nlp}
\nonumber\mbox{Minimize} & f(x),&\\ 
\mbox{s.t.} & h_i(x)=0, &i=1,\dots,p,\\
\nonumber&g_j(x)\leq0, &j=1,\dots,q,
\end{eqnarray}
where $f, h_{i}, g_{j}\colon \mathbb{R}^{n} \rightarrow \mathbb{R} $ are continuously differentiable functions. We denote by $A(x^*) := \{j \in \{1,\ldots, q\}\mid g_{j}(x^*) = 0\}$, the set of indices of active inequality constraints at a feasible point $x^*$.

It is well known that at a local minimizer $x^*$, it holds that $-\nabla f(x^*)\in\mathcal{T}(x^*)^\circ$, where $\mathcal{T}(x^*)$ denotes the (Bouligand) tangent cone to the feasible set at $x^*$ (see, e.g., \cite[Theorem 12.8]{nocedal}). However, since the tangent cone is a geometric object, this necessary optimality condition is not always easy to manipulate. For this reason, one considers the linearized cone, which is defined as follows:
$$
\mathcal{L}(x^*) := \left\{ d\in\R^n\mid
\nabla h_{i}(x^*)^{T}d = 0, i = 1,\ldots,p; \: \nabla g_{j}(x^*)^{T}d \leq 0, j \in A(x^*)\right\}.
$$
Its polar may be computed via Farkas' Lemma, obtaining:
$$
\mathcal{L}(x^*)^{\circ} = \left\{v \in \R^{n} \:\left| \:v = \sum_{i=1}^{p}\lambda_{i} \nabla h_{i}(x^*) + \sum_{j \in A(x^*)} \mu_{j} \nabla g_{j}(x^*), \mu_{j} \geq 0, j \in A(x^*) \right. \right\}.
$$
Hence, when $\mathcal{T}(x^*)^\circ=\mathcal{L}(x^*)^{\circ}$, this geometric optimality condition takes the form of the usual, much more tractable, Karush-Kuhn-Tucker conditions. Vectors $(\lambda_{i},\mu_{j})$ above are called Lagrange multipliers associated with $x^*$, and the set of all these vectors is denoted by $\Lambda(x^*)$ in this manuscript. 

A constraint qualification (CQ) is a condition that ensures the equality $\mathcal{T}(x^*)^\circ=\mathcal{L}(x^*)^{\circ}$. One of the most used CQ in the NLP literature is the well-known Linear Independence Constraint Qualification (LICQ), which states the linear independence of the set of gradients $\{\nabla h_{i}(x^*)\}_{i=1}^{p}\cup\{\nabla g_{j}(x^*)\}_{j \in A(x^*)}$. LICQ ensures not only the existence, but also the uniqueness of the Lagrange multiplier (see, e.g., \cite[Section 12.3]{nocedal}). 
Several weaker CQs have been defined for NLP. In this paper, we are interested in constant rank-type ones as first introduced by Janin in \cite{crcq}. Recall that in the NLP setting, we say that the Constant Rank Constraint Qualification (CRCQ) holds at a feasible point $x^*$ if there exists a neighborhood $V$ of $x^*$, such that for every subsets $\mathit{I} \subseteq \{1,\ldots,p\}$ and $\mathit{J} \subseteq A(x^*)$, the rank of $\{ \nabla h_{i}(x), \nabla g_{j}(x); i \in \mathit{I}, j \in \mathit{J}\}$ remains constant for all $x \in V$. CRCQ is clearly weaker than LICQ. 

Note that requiring only constant rank of the full set of gradients $\{\nabla h_{i}(x)\}_{i=1}^{p}\cup\{\nabla g_{j}(x)\}_{j \in A(x^*)}$ (which is known as the Weak Constant Rank (WCR) property) is not a CQ, as shown in \cite{ams2}. The necessity of considering every subset of this set of gradients may be seen from the definition of the linearized cone. Indeed, given $d\in\mathcal{L}(x^*)$, the relevant index set of inequality constraints gradients is given by $J=J_d:=\{j\in A(x^*)\mid\nabla g_j(x^*)^Td=0\}$, which cannot be chosen in advance if one only considers the point $x^*$. However,
this suggests that there is no need to consider subsets of indices for the equality constraints, that is,  it is enough to fix $I=\{1,\dots,p\}$. This condition, called Relaxed-CRCQ (RCRCQ), has been shown to be a CQ in~\cite{rcrcq}. This condition reads as follows: RCRCQ holds at a feasible point $x^*$ if there exists a neighborhood $V$ of $x^*$, such that for every subset  $\mathit{J} \subseteq A(x^*)$, the rank of $\{ \nabla h_{i}(x), \nabla g_{j}(x); i \in \{1,\ldots, p\}, j \in \mathit{J}\}$ remains constant for all $x \in V$.

These conditions can be seen as {\it constant linear dependence} conditions and thus it is natural to weaken these definitions by considering only {\it constant positive linear dependence}, providing conditions CPLD \cite{cpld} and its relaxed variant RCPLD \cite{rcpld}, both strictly weaker than Mangasarian-Formovitz CQ. This will be the most natural formulation for the CQs we propose in this paper. We refer the reader to \cite{rcpld}.

It turns out that the idea behind the construction of RCRCQ can be also extended to inequality constraints, providing an even weaker CQ. One seeks at characterizing a single index set $J$ which is relevant of having the constant rank property. This set consists of the indices of gradients defining the subspace component of $\mathcal{L}(x^*)^\circ$,  which is given by its lineality space.  More precisely, the lineality space of $\mathcal{L}(x^*)^\circ$, defined as the largest  linear space contained in  $\mathcal{L}(x^*)^\circ$,  is in this case given by $\mathcal{L}(x^*)^\circ\cap-\mathcal{L}(x^*)^\circ$. So, a gradient $\nabla g_{j}(x^*)$ belongs to $\mathcal{L}(x^*)^{\circ} \cap -\mathcal{L}(x^*)^{\circ}$ if, and only if, $-\nabla g_{j}(x^*) \in \mathcal{L}(x^*)^{\circ}$. Thus, for $J={J}_{-}(x^*) := \{j \in A(x^*)\mid -\nabla g_{j}(x^*) \in \mathcal{L}(x^*)^{\circ} \}$, we say that the Constant Rank of the Subspace Component (CRSC) CQ holds at a feasible point $x^*$ if there exists a neighborhood $V$ of $x^*$, such that the rank of $\{ \nabla h_{i}(x), \nabla g_{j}(x); i \in \{1,\ldots, p\}, j \in \mathit{J}_{-}(x^*)\}$ remains constant for all $x \in V$. It was proved in \cite{cpg} that CRSC is sufficient for the existence of Lagrange multipliers at a local minimizer, and this is the weakest of the CQs we have discussed.

CQ conditions discussed above in the NLP context have multiple applications. For instance, RCRCQ was used to compute the derivative of the value function in~\cite{minch},  as well as to prove the convergence of a second-order augmented Lagrangian algorithm to second-order stationary points in \cite{akkt2}. RCPLD and CRSC were shown to be sufficient for proving first-order global convergence of several algorithms while also implying the validity of an error bound property (cf. \cite{cpg}). Noteworthy, under CRSC, all inequality constraints in the set $J_{-}(x^*)$ behave locally as equality constraints, in the sense that they are active at any feasible point in a neighborhood of $x^*$.
Therefore, we strongly believe that the extension of these notions to a conic framework may have a major impact in stability and algorithmic theory for conic programming.

\section{Constraint qualifications conditions in second-order cone programming\label{sec:CQ-SOCP}}

%
%
 

Let us consider the second-order cone programming (SOCP) problem as follows:
\begin{eqnarray}\label{socp}
\nonumber\mbox{Minimize} & f(x),&\\ 
\mbox{s.t.} & h_i(x)=0, &i=1,\dots,p,\\
\nonumber&g_j(x)\in K_{m_j}, &j=1,\dots,\ell,
\end{eqnarray}
where the functions are continuously differentiable and the second-order cones are denoted by $K_{m_{j}}:=\{(z_0,\overline{z}) \in \R \times \R^{m_{j}-1}\mid z_0 \geq \|\overline{z}\| \}$ when $m_j>1$, and $K_{m_{j}} := \R_{+}$ (non-negative reals) otherwise. 

We say that the Karush-Kuhn-Tucker (KKT) conditions hold for problem (\ref{socp}) at a feasible point $x^*$ if there exists $\lambda\in\R^p$, $ \mu_j \in K_{m_j}$, $j=1,\dots,\ell,$ such that
\begin{eqnarray}\label{kkt1} 
\nabla_x L(x^*,\lambda, \mu ) = \nabla f(x^*) +J_h(x^*)^T\lambda -  \sum_{j=1}^\ell J_{g_j}(x^*)^T \mu_j =0 , \\
\langle \mu_j, g_j(x^*) \rangle = 0, \;\;\; j =1, \ldots, \ell. \label{kkt2}
\end{eqnarray}
Here, $L(x, \lambda, \mu) := f(x) + \langle \lambda, h(x) \rangle -  \sum_{j=1}^\ell \langle \mu_j, g_j(x) \rangle$ is the standard Lagrangian function for problem (\ref{socp}), and $\nabla_x L(x,\lambda, \mu )$ denotes the gradient of $L$ at $(x,\lambda, \mu )$ with respect to $x$. As usual, the set of all Lagrange multipliers $(\lambda, \mu )$ associated with the feasible point $x^*$, such that \eqref{kkt1}--\eqref{kkt2} are fulfilled, is denoted by $\Lambda(x^*)$.

As in NLP, one needs to assume a suitable CQ in order to ensure the existence of Lagrange multipliers associated with a local minimizer. In what follows, we recall the elements needed to define these CQs in the SOCP context.

The topological interior of $K_{m_{j}}$, denoted by $\inte(K_{m_{j}})$, and the non-zero boundary, denoted by $\bdp(K_{m_{j}})$, are respectively defined by
\begin{align*}
\inte (K_{m_{j}}) &:= \{(z_{0}, \overline{z}) \in \R \times \R^{m_{j}-1}\mid z_{0} > \|\overline{z}\|\},\\	
\bdp(K_{m_{j}}) &:= \{(z_{0}, \overline{z}) \in \R \times \R^{m_{j}-1}\mid z_{0} = \|\overline{z}\|> 0\}.	
\end{align*}
Thus, given a feasible point $x^*$, we introduce the index sets:
\begin{align*}
I_{int}(x^{*}) &:= \{j\ \in \{1,\ldots,\ell\} \mid g_j(x^{*}) \in \inte (K_{m_j})\},\\
I_B(x^{*}) &:=\{j \in \{1,\ldots,\ell\} \mid g_j(x^{*}) \in \bdp(K_{m_j})\},\\
I_0(x^{*}) &:= \{j \in \{1,\dots,\ell\} \mid g_j(x^*)=0\}.
\end{align*}

Moreover, the complementarity condition \eqref{kkt2} can be equivalently written as
\begin{equation} \label{kkt2circ}
 \mu_j\circ g_j(x^*) = 0, \;\;\; j =1, \ldots, \ell,
 \end{equation}
where the operation $\circ$ is defined for any couple of vectors $y := (y_0,\bar{y})$ and $s := (s_0,\bar{s})$, with the same dimension, as follows:
$$y\circ s := \left(
\begin{matrix}
\langle y , s \rangle \\  y_0 \bar s + s_0 \bar y
\end{matrix}
\right).
$$
For more details about this operation, its algebraic properties and its relation with Jordan algebras, see \cite[Section 4]{AG} and references therein. 

From \eqref{kkt2circ}, it is easy to check that complementarity condition is equivalently written in terms of the above-mentioned index sets as follows:
\begin{equation} \label{mu_complem}
 \mu_j= 0 \mbox{ if } j\in I_{int}(x^{*}), \quad  \mu_j= \alpha_j R_{m_j} g_j(x^*), \mbox{ for some  }\alpha_j \geq 0, \mbox{ if } j\in I_B(x^{*}) ,
 \end{equation}
and no condition on $\mu_j$ can be inferred when $j \in I_0(x^{*})$. Here, 
$R_{m}$ is an $m\times m$ diagonal matrix whose first entry is $1$ and the remaining ones are $-1$. 
\if{
Here, $J_{g_j}(x)$ denotes the $m_j\times n$ Jacobian of $g_j$ and $R_{m_j} = \left[ \begin{array}{cc}
1 & 0
^T \\
0 & - I_{m_{j-1}} 
\end{array}
\right]$, where $I_{m_{j}}$ is the $m_{j}$-dimensional identity matrix.
}\fi
Consequently, KKT conditions at $x^*$ can be characterized as the existence of $\lambda\in\R^p$, $ \mu_j \in K_{m_j}$, $j\in I_0(x^*)$, and $\alpha_j\geq0, j\in I_B(x^*)$, such that
\begin{eqnarray}\label{grakkt} 
\nabla f(x^*) +J_h(x^*)^T\lambda -  \sum_{j\in I_0(x^*)} J_{g_j}(x^*)^T \mu_j - \sum_{j\in I_B(x^*)}\alpha_j\nabla\phi_j(x^*) = 0,   
\end{eqnarray}
where
\[
\phi_j(x):=\half([g_j(x)]_0^2-\|\overline{g_j(x)}\|^2) \quad \mbox{for all } j\in I_B(x^*).
\]
Indeed, it is straightforward to check that $\nabla\phi_j(x)=J_{g_j}(x)^TR_{m_j}g_j(x)$ and multipliers $\mu_j$ for all $j\not\in  I_0(x^*)$ are recovered from \eqref{mu_complem}. 



The use of mappings $\phi_j$ is a consequence of applying the reduction approach to problem \eqref{socp}.  Actually, condition \eqref{grakkt} is simply KKT conditions at point $x^*$ for a locally equivalent version of problem \eqref{socp} for which constraints $g_j(x)\in K_{m_j}$ are replaced by $\phi_j(x)\geq 0$ when $j \in I_B(x^*)$,  and are omitted when $j \in I_{int}(x^*)$. For the sake of completeness, this reduced equivalent problem is explicitly stated here below:
\begin{eqnarray}\label{socp-reduced}
\nonumber\mbox{Minimize} & f(x),&\\ 
\mbox{s.t.} & h_i(x)=0, &i=1,\dots,p,\\
\nonumber&g_j(x)\in K_{m_j}, &j\in I_0(x^*),\\
\nonumber&\phi_j(x)\geq 0, &j \in I_B(x^*).
\end{eqnarray}

Despite its apparent simplicity in the SOCP setting, the reduction approach is a key tool in conic programming. It permits obtaining first- and second-order optimality conditions, to simplify some well-known CQs,  among other crucial properties.  See \cite[Section 3.4.4]{bonnans-shapiro} and \cite[Section 4]{BonRam} for more details. 
Throughout this article we will use KKT condition \eqref{grakkt} and problem \eqref{socp-reduced} to adapt CQ conditions from NLP to the SOCP setting \eqref{socp}. 

One of the most used (and strong) conditions to guarantee the existence of a Lagrange multiplier at a local minimizer $x^*$ is the nondegeneracy condition. Thanks to the reduction approach (cf. \cite[Equation 4.172]{bonnans-shapiro}), this condition can be equivalently defined as follows:
\begin{definition}
Let $x^*$ be a feasible point of \eqref{socp}. Consider all the row vectors of the matrices $J_h(x^*)$ and $J_{g_j}(x^*), j\in I_0(x^*)$ together with the row vectors $\nabla\phi_j(x^*)^T, j\in I_B(x^*)$. We say that \emph{nondegeneracy} holds at $x^*$ when these vectors are linearly independent.
\end{definition}

The nondegeneracy condition implies the existence and uniqueness of a Lagrange multiplier at a local minimizer $x^*$, and the reciprocal is true provided that $(x^*,\lambda,\mu)$ (with $(\lambda,\mu)\in\Lambda(x^*)$) is strictly complementary, that is, $g_j(x^*) +\mu_j \in  \inte (K_{m_j})$ for all $j=1,\dots,\ell$; see \cite[Proposition 4.75]{bonnans-shapiro}. Thus, nondegeneracy is the analogue of LICQ from nonlinear programming.
Note that there are other definitions of nondegeneracy e.g. \cite[Definition 18]{AG} and \cite[Definition 16]{BonRam}. However, all these definitions coincide in the case of SOCP problem \eqref{socp}. We address the reader to \cite[Section 4]{BonRam} for more details about nondegeneracy in the context of SOCP.

As LICQ in NLP, nondegeneracy condition is often considered too strong.  
For this reason, one typically assumes a weaker condition, called Robinson's CQ, 
which was originally defined in  \cite{Rob76} for a general conic setting. 
In our SOCP setting, we can use characterizations given in \cite[Proposition 2.97, Corollary 2.98 and Lemma 2.99]{bonnans-shapiro} to 
obtain the following equivalent definition: 
\begin{definition}Let $x^*$ be a feasible point of \eqref{socp}. We say that \emph{Robinson's CQ} holds at~$x^*$~if
\begin{equation}\label{eq:Rob-MF}
\begin{split}
J_h(x^*)^T\lambda +  \sum_{j=1}^\ell J_{g_j}(x^*)^T \mu_j =0 \mbox{ and }  \lambda\in\R^m, \, \mu_j \in K_{m_j}, \, \langle \mu_j, g_j(x^*)\rangle = 0,  j=1,\dots,\ell\\
\Rightarrow \:\: \lambda =0 \mbox{ and } \mu_j=0, \, j=1,\dots,\ell.
\end{split}
\end{equation}
\end{definition}

As in NLP, when $x^*$ is assumed to be a local solution of  \eqref{socp},
Robinson's CQ  \eqref{eq:Rob-MF} is equivalent to saying that the set of Lagrange multipliers $\Lambda(x^*)$ is nonempty and compact (cf. \cite[Props. 3.9 and 3.17]{bonnans-shapiro}). In this sense, condition \eqref{eq:Rob-MF} can be seen as an extension of Mangasarian-Fromovitz CQ in NLP to the SOCP setting \eqref{socp},  written in a dual form.

Thanks to \eqref{mu_complem}, condition  \eqref{eq:Rob-MF} can be rewritten as follows: 
\begin{equation}\label{eq:Rob}
\begin{split}
J_h(x^*)^T\lambda+\sum_{j\in I_0(x^*)}J_{g_j}(x^*)^T\mu_j + \sum_{j\in I_B(x^*)}\alpha_j\nabla \phi_j(x^*)=0,\\
\lambda\in\R^m, \mu_j\in K_{m_j}, j\in I_0(x^*); \: \alpha_j\geq0, j\in I_B(x^*)\\
\Rightarrow \:\: \lambda=0, \mu_j=0, j\in I_0(x^*); \: \alpha_j=0, j\in I_B(x^*).
\end{split}
\end{equation}
As we will see in the forthcoming sections, condition \eqref{eq:Rob} best fits our analysis. 

Note that \eqref{eq:Rob} can be interpreted as a conic linear independence of the (transposed) Jacobians and gradients involved in its definition. Indeed, given some finite number of convex and closed cones $C_j$ and denoting by $\prod_j C_j$ the cartesian product of these sets, we say that a correspondent set of matrices $V_j$ of appropriate dimensions is $\prod_j C_{j}$-linearly independent if
$$\sum_j V_j s_j = 0 \mbox{ and }  - s_j \in C_{j}^\circ  \mbox{ for all } j \:\:  \Rightarrow  \:\: s_j = 0  \mbox{ for all } j.$$
Then, \eqref{eq:Rob} coincides with the $\{0_p\}\times \prod_{j\in I_0(x^*)} K_{m_j}\times \R^{|I_B(x^*)|}_+$-linear independence of matrices: $J_h(x^*)^T$,  $J_{g_i}(x^*)^T$ with $ j\in I_0(x^*)$, and $\nabla \phi_j(x^*)$ with $j\in I_B(x^*)$. Here, $0_p$ denotes the null vector in $\R^p$. 
Moreover, when $C_{j}=\R_+$ for all $j$ in the definition above (and consequently, each matrix $V_j$ is simply a column vector), $\prod_j C_{j}$-linear independence coincides with the well-known positive linear independence. Then, condition~\eqref{eq:Rob} reminds the characterization of Mangasarian-Fromovitz CQ condition given by the positive linear independence of the gradients of active constraints (after replacing each equality constraint $h_i(x)=0$ by two inequalities $h_i(x)\geq 0$ and  $h_i(x)\leq 0$).  It is also interesting to note that $\{0_p\} \times \prod_{j=1,\dots,\ell} K_{m_j}$-linear independence of matrices $J_h(x^*)^T$  and $J_{g_i}(x^*)^T$ with $j=1,\dots,\ell$, is strictly stronger than Robinson's CQ \eqref{eq:Rob-MF}. This again shows how useful is the reduction approach for our analysis. 
Given the analyzed above, when Robinson's CQ fails, we say that the corresponding matrices in \eqref{eq:Rob} are conic linearly dependent. 

\section{The Approximate-KKT approach}

For the nonlinear programming problem \eqref{nlp}, the following \emph{Approximate-KKT} (AKKT) necessary optimality condition~\cite{akkt} is well known:

\begin{theorem}Let $x^*$ be a local minimizer of \eqref{nlp}. Then, there exist sequences $\{x^k\}\subset\R^n$, $\{\lambda^k\}\subset\R^p$, $\{\mu^k\}\subset\R^q_+$ such that $x^k\to x^*$ and 
\begin{equation}
\label{akkt-thm}
\nabla f(x^k)+\sum_{i=1}^p\lambda_i^k\nabla h_i(x^k)+\sum_{j\in A(x^*)}\mu_j^k\nabla g_j(x^k)\to0.
\end{equation}
\end{theorem}

We define $\mu_j^k\to0$ (or, equivalently, $\mu_j^k=0$) for $j\not\in A(x^*)$. Note that this does not require any constraint qualification at all and the sequence of approximate Lagrange multipliers $\{(\lambda^k,\mu^k)\}$ may be unbounded. If the sequence has a bounded subsequence, one may take a convergent subsequence such that the KKT conditions hold. In the unbounded case, one may define $M^k:=\max\{|\lambda_i^k|, i=1,\dots,p; \mu_j^k, j\in A(x^*)\}\to+\infty$ and divide the expression in \eqref{akkt-thm} by $M^k$. Thus, one may take an appropriate subsequence such that 
$$\frac{\lambda^k}{M^k}\to\lambda\in\R^p \quad \mbox{ and }\quad \frac{\mu_j^k}{M^k}\to\mu_j\geq0, \: j\in A(x^*),$$
obtaining the existence of  scalars $\lambda_i, i=1,\dots,p; \mu_j\geq 0,\, j\in A(x^*)$, not all equal to zero, satisfying
$$\sum_{i=1}^p\lambda_i\nabla h_i(x^*)+\sum_{j\in A(x^*)}\mu_j\nabla g_j(x^*)=0.$$
That is, the gradients of equality constraints and active inequality constraints are positive linearly dependent. This provides a simple proof for the existence of Lagrange multipliers under the Mangasarian-Fromovitz CQ (MFCQ). A very similar argument shows that the set of Lagrange multipliers at $x^*$ is bounded if, and only if, MFCQ holds.

\if{
In the case of the second-order cone programming \eqref{socp}, an extension of the AKKT necessary optimality condition has been proved in \cite{psocp} as follows:

\begin{theorem}Let $x^*$ be a local minimizer of \eqref{socp}. Then, there exist sequences $\{x^k\}\subset\R^n$, $\{\lambda^k\}\subset\R^m$, $\{\mu_i^k\}\subset K_{m_i}, i\in I_0(x^*)$, $\{\alpha_j^k\}\subset\R_+, j\in I_B(x^*)$ such that $x^k\to x^*$ and 
\begin{equation}
\label{akkt-socp}
\nabla f(x^k)+J_h(x^k)^T\lambda^k-\sum_{i\in I_0(x^*)}J_{g_i}(x^k)^T\mu_i^k -\sum_{j\in I_B(x^*)}\alpha_j^k\nabla \phi_j(x^k)\to0.
\end{equation}
\end{theorem}
}\fi

In order to go beyond MFCQ in nonlinear programming, one relies on the well-known \emph{Carath\'eodory's Lemma}, as stated in \cite{rcrcq}:

\begin{lemma}\label{lemma:carath} Let $v_1,\dots,v_{p+q}\in\mathbb{R}^n$ be such that $\{v_i\}_{i=1}^p$ are linearly independent. Consider scalars $\beta_i, i=1,\dots,p+q$, and denote $ y: = \sum_{i=1}^{p+q} \beta_{i} v_{i}$. Then,
there exist $J \subseteq  \{p+1, \ldots, p+q \}$ and scalars $\hat{\beta}_{i}, i \in \{1,\dots,p\}\cup J$, such that $ \{v_{i} \}_{i \in \{1,\dots,p\}\cup J} $ are linearly independent, $\beta_{i}>0$ implies $\hat{\beta}_i>0$, for all $i\in J$, and $y = \sum_{i \in \{1,\dots,p\}\cup J} \hat{\beta}_{i} v_{i}$.\end{lemma}

\medskip

Thus, in order to prove that CRCQ (and its weaker variants) is a CQ for the nonlinear programming problem \eqref{nlp}, we apply Carath\'eodory's Lemma to \eqref{akkt-thm}. This yields
$$
\nabla f(x^k)+\sum_{i\in I^k}\tilde{\lambda}_i^k\nabla h_i(x^k)+\sum_{j\in J^k}\tilde{\mu}_j^k\nabla g_j(x^k)\to0,
$$
with $I^k\subseteq\{1,\dots,p\}$, $J^k\subseteq A(x^*)$, $\tilde{\mu}_j^k\geq0, j\in J^k$,  and such that the vectors of the set $\{\nabla h_i(x^k)\}_{i\in I^k}\cup\{\nabla g_j(x^k)\}_{j\in J^k}$ are linearly independent for all $k$. Here, by the infinite pigeonhole principle and passing to a subsequence if necessary, index subsets $I^k$ and $J^k$ can be taken  as fixed and not depending on $k$. 
Then, the AKKT approach described above is similarly followed.
It is worth to emphasize here that the application of Carath\'eodory's Lemma preserves the sign of the candidate to multipliers, that is, $\tilde{\mu}_j^k$ has the same sign than $\mu_j^k$. This is a crucial step which is not clearly extended to the conic case (see \cite{errata}). Note that if $\{\nabla h_i(x^k)\}_{i=1}^p$ is linearly independent for all $k$, we may take $I_k=\{1,\dots,p\}$, which will be relevant in our analysis.

In the sequel, we will use the extension of the AKKT necessary optimality condition for second-order cone programming \eqref{socp}, as presented in \cite{psocp}:

\begin{theorem}Let $x^*$ be a local minimizer of \eqref{socp}. Then, there exist sequences $\{x^k\}\subset\R^n$, $\{\lambda^k\}\subset\R^p$, $\{\mu_j^k\}\subset K_{m_j}, j\in I_0(x^*)$, $\{\alpha_j^k\}\subset\R_+, j\in I_B(x^*)$ such that $x^k\to x^*$ and 
\begin{equation}
\label{akkt-socp}
\nabla f(x^k)+J_h(x^k)^T\lambda^k-\sum_{j\in I_0(x^*)}J_{g_j}(x^k)^T\mu_j^k -\sum_{j\in I_B(x^*)}\alpha_j^k\nabla \phi_j(x^k)\to0.
\end{equation}
\end{theorem}

%
%
\section{A proposal of constraint qualifications for second-order cones\label{sec:CQ4SOCP}}

Following the previous discussion, we present a ``naive'' formulation of constant rank constraint qualifications for the second-order cone programming problem \eqref{socp}. 

\begin{definition}\label{def:RCPLD} Let $x^*$ be a feasible point of problem \eqref{socp} and $I\subseteq\{1,\dots,p\}$ be such that $\{\nabla h_i(x^*)\}_{i\in I}$ is a basis of the linear space generated by vectors $\{\nabla h_i(x^*)\}_{i=1}^p$. We say that the \emph{ Relaxed Constant Positive Linear Dependence (RCPLD)} condition holds at~$x^*$ when, for all $J\subseteq I_B(x^*)$, there exists a neighborhood $V$ of $x^*$ such that:
\begin{itemize} 
\item $\{\nabla h_i(x)\}_{i=1}^p$ has constant rank for all $x$ in $V$;
\item if the system
\begin{eqnarray*}\sum_{i\in I}\lambda_i \nabla h_i(x^*)+\sum_{j\in I_0(x^*)}J_{g_j}(x^*)^T\mu_j+\sum_{j\in J}\alpha_j\nabla\phi_j(x^*)=0,\\
\lambda_i\in\R, i\in I; \:\: \mu_j\in K_{m_j}, j\in I_0(x^*); \:\: \alpha_j\geq0, j\in J,
\end{eqnarray*}
has a not all zero solution $(\lambda_i)_{i\in I}, (\mu_j)_{j\in I_0(x^*)}, (\alpha_j)_{j\in I_B(x^*)}$, then vectors $\{\nabla h_i(x)\}_{i\in I}\cup\{\nabla\phi_j(x)\}_{j\in J}$ are linearly dependent for all $x$ in $V$.
\end{itemize}
\end{definition}

\medskip

Note that Robinson's CQ implies RCPLD since it states the conic linear independence of the corresponding sets (and thus, for all its subsets) while RCPLD allows its conic linear dependence, as long as the linearly dependence is maintained for a reduced subset in a neighborhood.

The definition above takes into account our inability to relax Robinson's CQ for cones $K_{m_j}$ with $j\in I_0(x^*)$, as the linear dependence for $x$ near $x^*$ is required only for equalities and for constraints at the boundary. %
%
Indeed, note that in the case when $I_B(x^*)=\emptyset$ and no equalities are considered (i.e., $p=0$), RCPLD coincides with  Robinson's CQ \eqref{eq:Rob-MF}. This is an immediate consequence of the adopted convention that states that the empty set is always a linear independent set. On the other hand, we are aware that Definition~\ref{def:RCPLD} is unnecessarily strong when $m_j=1$ for an index $j\in I_0(x^*)$. Indeed, in such case, the associated inequality $g_j(x)\in K_{m_j}$ corresponds to an inequality constraint of the form $g_j(x)\geq 0$, which is active at $x^*$. Hence, RCPLD definition can be slightly modified to take this situation into account as follows: define $A(x^*):=\{ j\in I_0(x^*) \mid m_j=1\}$, and remove those indices from $ I_0(x^*)$, that is,
define $\tilde I_0(x^*):=I_0(x^*)\setminus A(x^*)$. Indices in $A(x^*)$ can thus be treated similarly to those in $I_B(x^*)$. So, by defining $\phi_j(x):=g_j(x)$ when $j\in A(x^*)$, a slightly weaker version of RCPLD can be obtained by replacing $I_0(x^*)$ by $\tilde I_0(x^*)$ and  $I_B(x^*)$ by  $I_B(x^*)\cup A(x^*)$ in Definition \ref{def:RCPLD}. Since this modification has no consequence in the proof of Theorem \ref{thm:RCPLD}, we do not include it in its statement.

The point raised in the last paragraph explains why Definition \ref{def:RCPLD} is considered a ``naive" extension of a constant rank-type condition. Before proving that RCPLD is a CQ for problem \eqref{socp}, we make further observations related to this point.

\begin{remark}
a)   When we choose $J=\emptyset$  in Definition \ref{def:RCPLD}, we necessarily obtain that there is no non-zero solution $(\lambda_i,\mu_j)$, with $i\in I$ and $j\in I_0(x^*)$, to the system:
 $$\sum_{i\in I}\lambda_i \nabla h_i(x^*)+\sum_{j\in I_0(x^*)}J_{g_j}(x^*)^T\mu_j=0 \quad \mbox{and} \quad \lambda_i\in\R, i\in I; \quad \, \mu_j\in K_{m_j}, j\in I_0(x^*).$$
 This is equivalent to saying that Robinson's CQ holds at $x^*$ for the constrained set $\Gamma_0:=\{ x \mid h_i(x)=0, \, i \in I, \, g_j(x)\in K_{m_j}, \, j\in I_0(x^*)\}$. So, RCPLD ensures that Robinson's CQ is fulfilled at $x^*$ for the active set $\Gamma_0$. Actually, by using the   slight modification discussed above, we can exclude standard nonlinear constraints from  $I_0(x^*)$, and conclude that it only implies the weaker condition:  Robinson's CQ holds at $x^*$ for the constrained set $\tilde \Gamma_0:=\{ x \mid h_i(x)=0, \, i \in I, \, g_j(x)\in K_{m_j}, \, j\in I_0(x^*), \, m_j >1\}$.
  
  b) Consider the case when problem \eqref{socp} reduces to NLP \eqref{nlp}, that is, $\tilde{I}_0(x^*)=\emptyset$ and $I_B(x^*)=\emptyset$. Then, RCPLD in Definition \ref{def:RCPLD} reduces to the respective definition for nonlinear programming \cite{rcpld}. In particular, by enlarging the system to include $\alpha_j\in\R, j\in J$, instead of only considering $\alpha_j\geq0, j\in J$, the definition reduces to an equivalent characterization  (see \cite{rcpld}) of RCRCQ:  $\{\nabla h_i(x)\}_{i=1}^p$ has constant rank for $x$ around $x^*$ and for all $J\subseteq A(x^*)$, if the set $\{\nabla h_i(x^*)\}_{i\in I}\cup\{\nabla\phi_j(x^*)\}_{j\in J}$ is linearly dependent, then $\{\nabla h_i(x)\}_{i\in I}\cup\{\nabla\phi_j(x)\}_{j\in J}$ must remain linearly dependent for all $x$ in a neighborhood of $x^*$ (here, the set $I$ is fixed as in Definition \ref{def:RCPLD}). The latter also explains why RCPLD, given in Definition \ref{def:RCPLD}, is considered a constant rank-type condition for problem \eqref{socp}.
  
  c) Differently from the definition of nondegeneracy and Robinson's CQ, the choice of the reduction function $\phi(\cdot)$ gives rise to different constant rank conditions. For instance, one could formulate a similar, but different, condition by considering the alternative reduction function $\tilde{\phi}_j(x):=[g_j(x)]_0-\|\overline{g_j(x)}\|$ for $j\in I_B(x^*)$. This is a well-known fact for nonlinear programming, which establishes that when a constraint set satisfies CRCQ, it can be  rewritten in such a way that it fulfills Robinson's CQ \cite{shulu}.
 \end{remark}

\begin{theorem} \label{thm:RCPLD}
Let $x^*$ be a feasible point of problem \eqref{socp} satisfying  the AKKT condition~\eqref{akkt-socp} and RCPLD. Then, the KKT conditions hold at $x^*$. In particular, RCPLD is a constraint qualification.
\end{theorem}
\begin{proof}
AKKT condition \eqref{akkt-socp} ensures the existence of  sequences $\{x^k\}\subset\R^n$, $\{\lambda^k\}\subset\R^p$, $\{\mu_j^k\}\subset K_{m_j}, j\in I_0(x^*)$, $\{\alpha_j^k\}\subset\R_+, j\in I_B(x^*)$, such that $x^k\to x^*$ and 
\begin{equation*}
\nabla f(x^k)+\sum_{i=1}^p\lambda_i^k\nabla h_i(x^k)-\sum_{j\in I_0(x^*)}J_{g_j}(x^k)^T\mu_j^k -\sum_{j\in I_B(x^*)}\alpha_j^k\nabla \phi_j(x^k)\to0.
\end{equation*}
By the constant rank assumption on the equality constraints, and the definition of $I$, we may rewrite $\sum_{i=1}^p\lambda_i^k\nabla h_i(x^k)=\sum_{i\in I}\tilde{\lambda}_i^k\nabla h_i(x^k)$ for new scalars $\tilde{\lambda}^k_i\in\R, i\in I$,  such that vectors $\{\nabla h_i(x^k)\}_{i\in I}$ are linearly independent. Applying Carath\'eodory's Lemma, for each $k$, we get $J^k\subseteq I_B(x^*)$ and new scalars $\hat{\lambda}^k_i\in\R, i\in I$, $\hat{\alpha}^k_j\geq0, j\in J^k$, such that
\begin{equation}
\label{limit}
\nabla f(x^k)+\sum_{i\in I}\hat{\lambda}_i^k\nabla h_i(x^k)-\sum_{j\in I_0(x^*)}J_{g_j}(x^k)^T\mu_j^k -\sum_{j\in J^k}\hat{\alpha}_j^k\nabla \phi_j(x^k)\to0,
\end{equation}
and vectors $\{\nabla h_i(x^k)\}_{i\in I}\cup\{\nabla\phi_j(x^k)\}_{j\in J^k}$ are linearly independent. By the infinite pigeonhole principle, without loss of generality we can consider subsequences, which are renamed as the original ones, for which sets $J^k$ are the same for all $k$. This set is denoted by $J$.

Define $M^k:=\max\{|\hat{\lambda}_i^k|, i\in I; \|\mu_i^k\|, i\in I_0(x^*); \hat{\alpha}_j, j\in J\}$. If $\{M^k\}$ is bounded, any accumulation point of $\{\hat{\lambda}_i^k, i\in I; \mu_i^k , i\in I_0(x^*); \hat{\alpha}_j, j\in J\}$ (after replacing by 0 the values for indices that are neither in $I$, nor in $J$) satisfies \eqref{grakkt}. Hence, $x^*$ is a KKT point of \eqref{socp}. Otherwise, we may take a subsequence such that $M^k\to+\infty$, and divide the expression in~\eqref{limit} by $M^k$, considering convergent subsequences such that
\begin{eqnarray*}
  & & -\frac{\hat{\lambda}_i^k}{M^k}\to\lambda_i\in\R, \: i\in I; \qquad \frac{\mu_j^k}{M^k}\to\mu_j\in K_{m_j}, \: j\in I_0(x^*); \\
  & & \frac{\hat{\alpha}_j^k}{M^k}\to\alpha_j\geq0, \: j\in J, \qquad \mbox{ with }(\lambda_i,\mu_j,\alpha_j)\neq0,
\end{eqnarray*}
and obtaining
$$\sum_{i\in I}{\lambda}_i\nabla h_i(x^*)+\sum_{j\in I_0(x^*)}J_{g_j}(x^*)^T\mu_j +\sum_{j\in J}{\alpha}_j\nabla \phi_j(x^*)=0.$$ 
Then, since vectors $\{\nabla h_i(x^k)\}_{i\in I}\cup\{\nabla\phi_j(x^k)\}_{j\in J}$ are linearly independent, this  contradicts the definition of RCPLD.\jo{\hfill\qed}
\end{proof}

 Exact definition of RCPLD in nonlinear programming can be consulted in  \cite{rcpld}. The definition of CRCQ \cite{crcq}, RCRCQ \cite{minch}, and CPLD \cite{cpld} may be analogously extended. They are omitted.  We only introduce the extension  of CRSC \cite{cpg} for this SOCP setting, since its definition is more involving and differs from its nonlinear programming counterpart. For the sake of completeness, the definition of CRSC considers sets $\tilde I_0(x^*)$ and $A(x^*)$. To prove that CRSC is a CQ is enough to follow the proof of Theorem \ref{thm:RCPLD}, so it is omitted.

\begin{definition}
Let $x^*$ be a feasible point of \eqref{socp} and $J_{-}(x^*)\subseteq I_B(x^*)\cup A(x^*)$ be defined~as 
\begin{equation*}
\begin{split}J_{-}(x^*):=\Bigg\{j_0\in I_B(x^*)\cup A(x^*) {\Big |} -\nabla\phi_{j_0}(x^*)=\sum_{i=1}^p\lambda_i\nabla h_i(x^*)+\hspace{-5pt}\sum_{j\in I_B(x^*)\cup A(x^*)}\hspace{-5pt}\alpha_j\nabla\phi_j(x^*),\\
\mbox{ for some } \lambda_i\in\R, \alpha_j\geq0 \Bigg\}.
\end{split}\end{equation*}
Set $J_{+}(x^*):=I_B(x^*)\cup A(x^*)\backslash J_{-}(x^*)$. We also define $I\subseteq\{1,\dots,p\}$ and $J\subseteq J_{-}(x^*)$ such that $\{\nabla h_i(x^*)\}_{i\in I}\cup\{\nabla\phi_j(x^*)\}_{j\in J}$ is a basis of the linear space generated by $\{\nabla h_i(x^*)\}_{i=1}^p\cup\{\nabla\phi_j(x^*)\}_{j\in J_{-}(x^*)}$. We say that the \emph{Constant Rank of the Subspace Component} (CRSC) condition holds at $x^*$ when there exists a neighborhood $V$ of $x^*$ such that:
\begin{itemize} 
\item $\{\nabla h_i(x)\}_{i=1}^p\cup\{\nabla\phi_j(x)\}_{j\in J_{-}(x^*)}$ has constant rank for all $x$ in $V$;
\item the system
\begin{eqnarray*}\sum_{i\in I}\nabla h_i(x^*)\lambda_i+\sum_{j\in \tilde I_0(x^*)}J_{g_j}(x^*)\mu_j+\sum_{j\in J\cup J_{+}(x^*)}\nabla\phi_j(x^*)\alpha_j=0,\\
\lambda_i\in\R, i\in I; \quad \mu_j\in K_{m_j}, j\in \tilde I_0(x^*); \quad \alpha_j\in\R, j\in J; \quad \alpha_j\geq0, j\in J_{+}(x^*),
\end{eqnarray*}
has only the trivial solution.
\end{itemize}
\end{definition}

\medskip

Note that when $\tilde{I}_0(x^*)=\emptyset$, the second requirement in the definition of CRSC always holds \cite{cpg}.

As said above, both definitions, RCPLD and CRSC, are ``naive'' in the sense that they do not improve on Robinson's CQ regarding multi-dimensional cones at zero. That is, when all constraint indices belong to $\tilde{I}_0(x^*)$, both definitions coincide with Robinson's CQ \eqref{eq:Rob-MF}. However, the example below shows that RCPLD and CRSC are strictly weaker than Robinson's CQ:

\begin{example}\label{ex:socp}
Consider the constraint set defined by
$$g(x):=(g_0(x),g_1(x)):=(x,x)\in K_2,$$
where $x$ is one-dimensional. Clearly, $x^*=1$ is feasible and the single constraint is in the boundary, i.e. $I_B(x^*)$ is the only nonempty index set. Reduced constraint is such that $\phi(x):=\half (g_0(x)^2-g_1(x)^2)=0$ for all $x$. Then, it follows that $\nabla\phi(x^*)=0$ and consequently, Robinson's CQ fails. However, $\nabla\phi(x)=0$ for all $x$, which implies that RCPLD holds. CRSC also holds by noting that the reduced constraint belongs to the index set $J_{-}(x^*)$, whose gradient has constant rank, and $\tilde{I}_0(x^*)=\emptyset$, which is sufficient for ensuring the second condition. Indeed, $J=\emptyset$ is a basis for the linear space generated by the constraint gradient in $J_{-}(x^*)$ and the result follows by the linear independence of the empty set.
\end{example}

%
%

\section{Extension to semidefinite programming}

Consider the semidefinite programming (SDP) problem with multiple constraints:
\begin{eqnarray}\label{sdp}
\nonumber\mbox{Minimize} & f(x),&\\ 
\mbox{s.t.} & h(x)=0, &\\
\nonumber&g_j(x)\in \S^{m_j}_+, &j=1,\dots,\ell,
\end{eqnarray}
where $f:\R^n \to \R$, $h:\R^n\to \R^p,$ and  $g_j:\R^n \to \S^{m_j}$  are continuously differentiable functions, $\S^{m_j}$ is the linear space of $m_j\times m_j$ real symmetric matrices equipped with the inner product $A\cdot B := \tr(A B )$, where $\tr(AB)$ denotes the sum of the elements of the diagonal of $AB$ for all  matrices $A, B\in  \S^{m_j}$, and $$\S^{m_j}_+:=\{M\in \S^{m_j}\mid z^T M z\geq 0, \forall z\in \R^{m_j}\}$$ is the closed convex cone of all positive semidefinite elements of $\S^{m_j}$, for all $j=1,\ldots,\ell$. We denote by $\preceq_j$ the partial order relation induced by $\S^{m_j}_+$, that is, $A\preceq_j B$ if, and only if, $B-A\in \S^{m_j}_+$. For the sake of notation, the index $j$ is omitted throughout the paper and this relation order is simply denoted by $\preceq$. The order relations $\succeq$, $\succ$, and $\prec$ are similarly defined.

We end this subsection by recalling the Karush-Kuhn-Tucker conditions in the SDP framework. We say that KKT conditions hold at a feasible point $x^*$ of problem \eqref{sdp} when there exist Lagrange multipliers $\lambda\in \R^p$ and $\mu_j\in \S^{m_j}$, $j=1,\ldots,\ell$ such that
\begin{subequations}
\label{eq:KKT_P}
\begin{align}
\nabla f(x^*) +J_h(x^*)^T\lambda -  \sum_{j=1}^\ell J_{g_j}(x^*)^T \mu_j,\label{eq:KKT_P_1}\\
g_j(x^*)\cdot \mu_j = 0, \ j=1,\ldots,\ell,\label{eq:KKT_P_2}
\end{align}
\end{subequations}
with  
\begin{equation}\nonumber \label{def:adjoint}
J_{g_j}(x^*)^T z:=(\partial_1 g_j(x^*)\cdot z,\ldots,\partial_n g_j(x^*)\cdot z)^T, \quad \forall z\in \S^{m_j},
\end{equation}
where $\partial_i g_j(x^*)$ is the partial derivative of $g_j$ with respect to the variable $x_i$, at $x^*$, for each $i=1,\ldots,n$. In fact, $J_{g_j}(x^*)^T$ is the adjoint of the linear mapping $J_{g_j}(x^*)$, defined by $$J_{g_j}(x^*)d:=\sum_{i=1}^n d_i \partial_i g_j(x^*),$$ for all $d=(d_1,...,d_n)^T\in \R^n$, $j=1,\ldots,\ell$.

\subsection{Revisiting constraint qualifications for multifold SDP}

Constraint qualification conditions recalled in Section \ref{sec:CQ-SOCP} for SOCP have been also well established for SDP problem \eqref{sdp}. In this section, we start by quickly recalling Robinson's CQ, before proceeding with the study of nondegeneracy condition, which needs more attention for our purposes.

As in the SOCP setting, Robinson's CQ \cite{Rob76}  can be equivalently characterized via the properties established in  \cite[Proposition 2.97, Corollary 2.98 and Lemma 2.99]{bonnans-shapiro} in its dual form: 
\begin{definition}
We say that \emph{Robinson's CQ} holds at a feasible point $x^*$ of problem \eqref{sdp} when
\begin{equation}\label{def:RobinsonSDP}
\left.
\begin{aligned}
J_{h}(x^*)^T\lambda +\sum_{j=1}^\ell J_{g_j}(x^*)^T \mu_j=0,\\
\quad g_j(x^*)\cdot \mu_j=0, \ \forall j=1,\ldots,\ell,\\
\quad \mu_j\in \S^{m_j}_+, \ \forall j=1,\ldots,\ell, 
\end{aligned}
\right\}
\quad \Rightarrow \quad \mu_j=0, \ \forall j=1,\ldots,\ell.
\end{equation}
\end{definition}

\if{Robinson's CQ \eqref{def:RobinsonSDP} is known to be equivalent to the following condition (see e.g. \cite[Prop. 2.97]{bonnans-shapiro}):
\begin{equation}\label{eq:charact_Robinson}
\Im J_h(x^*)\times \prod_{j=1}^\ell \left(\Im J_{g_j}(x^*)+T_{\mathbb{S}^{m_j}_+} (g_j(x^*))\right) = \R^p \times \prod_{j=1}^\ell \mathbb{S}^{m_j},
\end{equation}
where $\prod$ denotes a Cartesian product. 
}\fi

As in SOCP, Robinson's CQ is considered as the natural extension of Mangasarian-Fromovitz CQ from NLP to the SDP setting. Actually, when $x^*$ is assumed to be a local solution of  \eqref{socp},
Robinson's CQ  \eqref{def:RobinsonSDP} is equivalent to saying that the set of Lagrange multipliers $\Lambda(x^*)$ is nonempty and compact (cf. \cite[Props. 3.9 and 3.17]{bonnans-shapiro}). 

 \if{
 
Thanks to \eqref{mu_complem}, condition  \eqref{eq:Rob-MF} can be rewritten as follows: 
\begin{equation}\label{eq:Rob}
\begin{split}
J_h(x^*)^T\lambda+\sum_{j\in I_0(x^*)}J_{g_j}(x^*)^T\mu_j + \sum_{j\in I_B(x^*)}\alpha_j\nabla \phi_j(x^*)=0,\\
\lambda\in\R^m, \mu_j\in K_{m_j}, j\in I_0(x^*); \: \alpha_j\geq0, j\in I_B(x^*)\\
\Rightarrow \:\: \lambda=0, \mu_j=0, j\in I_0(x^*); \: \alpha_j=0, j\in I_B(x^*).
\end{split}
\end{equation}
As we will see in the forthcoming sections, condition \eqref{eq:Rob} best fits our analysis. 

Note that \eqref{eq:Rob} can be interpreted as a conic linear independence of the (transposed) Jacobians and gradients involved in its definition. Indeed, given some finite number of convex and closed cones $C_j$ and denoting by $\prod_j C_j$ the cartesian product of these sets, we say that a correspondent set of matrices $V_j$ of appropriate dimensions is $\prod_j C_{j}$-linearly independent if
$$\sum_j V_j s_j = 0 \mbox{ and }  - s_j \in C_{j}^\circ  \mbox{ for all } j \:\:  \Rightarrow  \:\: s_j = 0  \mbox{ for all } j.$$
Then, \eqref{eq:Rob} coincides with the $\{0_p\}\times \prod_{j\in I_0(x^*)} K_{m_j}\times \R^{|I_B(x^*)|}_+$-linear independence of matrices: $J_h(x^*)^T$,  $J_{g_i}(x^*)^T$ with $ j\in I_0(x^*)$, and $\nabla \phi_j(x^*)$ with $j\in I_B(x^*)$. Here, $0_p$ denotes the null vector in $\R^p$.

XX

}\fi

Let us now recall nondegeneracy condition in the SDP context. The notion of nondegeneracy (called transversality therein) was introduced by Shapiro and Fan in \cite[Section 2]{ShapiroFan} by means of tangent spaces in the context of eigenvalue optimization. An equivalent form is proven in \cite[Equation (4.172)]{bonnans-shapiro} for  reducible cones. This is adopted as a formal definition in our multifold SDP setting:

\begin{definition} \label{def:NDG}
We say that a feasible point $x^*$ of problem \eqref{sdp} is \emph{nondegenerate} when the following relation is satisfied
\begin{equation}\label{eq:NDG}
\Im \, {\cal A}(x^*)  + \{0\} \times \prod_{j=1}^\ell \lin(T_{\mathbb{S}^{m_j}_+} (g_j(x^*))) = \R^p \times \prod_{j=1}^\ell \mathbb{S}^{m_j},
\end{equation}
where 
$$ {\cal A}(x^*) := 
\left( \begin{matrix} J_h(x^*)   \\ J_{g_j}(x^*) ; \,  j=1,..., \ell \end{matrix} \right)  $$
is a linear mapping from $ \R^n$ to $\R^p \times \prod_{j=1}^\ell \mathbb{S}^{m_j} $.
\end{definition}

As it happens in SOCP, the nondegeneracy condition is considered to be a natural analogue of LICQ from NLP to SDP. 
Actually, nondegeneracy condition \eqref{eq:NDG} implies the existence and uniqueness of a Lagrange multiplier at a local minimizer $x^*$, and the reciprocal is true provided that $(x^*,\lambda,\mu)$ (with $(\lambda,\mu)\in\Lambda(x^*)$) is strictly complementary, that is, $g_j(x^*) +\mu_j \succ 0$ for all $j=1,\dots,\ell$; see \cite[Proposition 4.75]{bonnans-shapiro}.
However, this analogy only makes sense when matrix blocks $g_j(x^*)$ are chosen in a ``minimal" way, in the sense of avoiding zeros in the off diagonal entries. In particular, an NLP problem with $\ell$ inequality constraints should be modeled as an instance of \eqref{sdp} with $m_1=\ldots=m_\ell=1$. Only in that case, nondegeneracy coincides LICQ. To stress the point above, we recall here below some results from \cite[Section 5]{BonRamSDP}. 

Consider the NLP problem of minimizing $f(x)$ under two constraints: $g_1(x)\geqslant 0$ and $g_2(x)\geqslant 0$, where $f, g_1$, and $g_2$ are smooth real-valued functions. Let $x^*$ be a local mimimun for which $g_1(x^*)=g_2(x^*)=0$ and LICQ holds (i.e., vectors $\nabla g_1(x^*)$ and $\nabla g_2(x^*)$ are linearly independent). Denote by $\bar{\mu}_1$ and $\bar{\mu}_2$ the unique associated Lagrange multipliers, and assume that strict complementarity holds: $\bar \mu_i >0$ for $i=1,2$. If this NLP problem is written as the following SDP problem 
\begin{eqnarray}\label{single-sdp}
\nonumber\mbox{Minimize} & f(x),\\ 
\mbox{s.t.} &
\begin{bmatrix}
g_1(x)&0\\
0& g_2(x)
\end{bmatrix}\in \S^{2}_+,
\end{eqnarray}
%
then nondegeneracy condition \eqref{eq:NDG} never holds. Indeed,
the Lagrange multiplier associated with $x^*$ for the reformulated problem \eqref{single-sdp} is never unique.
It is enough to note that the matrix 
$$
\bar \mu := 
\begin{bmatrix}
\bar\mu_1 & 0\\
0 & \bar\mu_2
\end{bmatrix}
$$
is an associated Lagrange multiplier as well as 
$$\bar{\mu} + t \left(\begin{matrix} 0 & 1 \\ 1 &0
\end{matrix}\right),$$
for any $t\in\R$ such that $t^2\leq {\bar{\mu}_1\bar{\mu}_2}$. Of course, this apparent inconsistency occurs not only for diagonal matrices but also for any SDP problem with a diagonal structure (see e.g. \cite[Lemma 5.1]{BonRamSDP}), and it is due to an inappropriate modeling decision regarding the sparse structure of the studied SDP problem.


On the other hand, this phenomenon does not occur with Robinson's CQ, which is always preserved independently of the block structure of the SDP constraint set. This may be one of the reasons why multifold SDP is not often taken into consideration in the literature, along with the fact that interior-point methods are knowingly capable of exploiting block-diagonal structure (see Gondzio's review~\cite{gondzioipm} and references therein for details). It is not expected, though, that every constraint qualification will be preserved between multifold and block-diagonal representations. In particular, the constraint qualifications we define in the next section are defined by means of exploiting the multifold structure. In this context, they are strictly weaker than Robinson's CQ, while if one considers a single block-diagonal representation our condition would resume to Robinson's CQ. Furthermore, since our analysis is related to AKKT sequences, which describe the output of many practical algorithms, our results provide a stronger convergence theory for them when applied to SDP problems under multifold representation.

\if{
\begin{remark}
Another relevant difference between \eqref{sdp} and \eqref{single-sdp} is that nondegeneracy for \eqref{sdp} implies $n\geq \sum_{j=1}^\ell (m_j-r_j)(m_j-r_j+1)/2$, which constraints the dimension of the space of matrices and of the kernel of $g_j(x^*)$, for each $j=1,\ldots,\ell$.  
However, nondgeneracy for \eqref{single-sdp} implies on a possibly more severe constraint over the dimensions: $n\geq (m-r)(m-r+1)/2$, where $m=m_1+\ldots+m_\ell$ and $r=r_1+\ldots+r_\ell$.
\end{remark}
}\fi

For more details about  the nondegeneracy condition in the semidefinite programming
context, see e.g. \cite{BonRamSDP,Shapiro}. In particular, Nondegeneracy condition for multifold SDP given in Definition \ref{def:NDG}  and the discussion above are inspired from  \cite[Section 5]{BonRamSDP}.

In the next section we propose a naive RCPLD condition similar to Definition~\ref{def:RCPLD} for multifold SDP, as in \eqref{sdp}. We note that CPLD has already been used in the context of SDP problems in \cite{chineses}, however, they consider the application of an augmented Lagrangian method for a mixed problem with SDP constraints and NLP constraints, where the NLP constraints are not penalized and are carried out to the subproblems. Hence, the usual CPLD is assumed for the NLP constrained subproblems, in the context of feasibility results, while Robinson's CQ is assumed for the full problem in the context of optimality results. In particular, no CPLD-type CQ is introduced for the full problem.

\subsection{A constant rank condition for SDP}

Denote the smallest eigenvalue of a matrix $A$ by $\sigma_{\min}(A)$ and its associated unitary eigenvectors by $\nu_{\min}(A)$ and $-\nu_{\min}(A)$. It is known that $\sigma_{\min}$ is continuously differentiable at $A$ when $\sigma_{\min}(A)$ is simple, i.e., when it has algebraic multiplicity equal to one, and that $J_{\sigma_{\min}}(A)=\nu_{\min}(A)\nu_{\min}(A)^T$ in this case (see, e.g., \cite{ShapiroFan}). So, given a local minimizer $x^*$, the composition $\sigma_{\min}\circ g_j$ is a reduction mapping for the block $j$ when $\sigma_{\min}(g_j(x^*))$ is simple, playing a similar role to $\phi_j(x)$ for problem \eqref{socp-reduced}. Also, in this scenario, 
\begin{equation}\label{eq:der-sigma_min} 
\nabla (\sigma_{\min}( g_j (x))=J_{g_j}(x)^T  J_{\sigma_{\min}}(g_j(x))
\end{equation}
when $x$ is close enough to $x^*$. 
This motivates us to define an analogue of problem \eqref{socp-reduced} for SDP as follows:
\begin{eqnarray}\label{sdp-reduced}
\nonumber\mbox{Minimize} & f(x),&\\ 
\mbox{s.t.} & h(x)=0,\\
\nonumber&g_j(x)\in \S^{m_j}_+, &j\in I_N(x^*),\\
\nonumber&\sigma_{\min}(g_j(x))\geq 0, &j \in I_R(x^*),
\end{eqnarray}
where $$I_R(x^*):=\{j\in\{1,\ldots,\ell\} \mid 0=\sigma_{\min}(g_j(x^*))  \textnormal{ is simple}\}$$ and $$I_N(x^*):=\{j\in\{1,\ldots,\ell\} \mid 0=\sigma_{\min}(g_j(x^*))  \textnormal{ is not simple}\}.$$ Note that \eqref{sdp-reduced} is locally equivalent to \eqref{sdp} and that we have removed for simplicity all the constraints such that 
$g_j(x^*) \succ 0$, i.e., the ``inactive'' ones, in the reformulated problem. 
However, in problem \eqref{sdp-reduced}, we have not applied the reduction approach to blocks $j\in I_N(x^*)$.
Roughly speaking, our approach consists of defining a constraint qualification that relaxes Robinson's CQ to a constant rank-type condition, but only at the constraints indexed by $I_R(x^*)$, which are the ones that are well-behaved enough to be fully replaceable by a single real-valued constraint. As in the SOCP case, our strategy for proving that this is indeed a constraint qualification is based on sequential optimality conditions.

In~\cite{AHV}, the AKKT condition was extended for SDP. Next, we present an adapted version of it for problems with mixed NLP and SDP constraints, like \eqref{sdp-reduced}:

\begin{theorem}\label{thm:akkt-sdp}Let $x^*$ be a local minimizer of \eqref{sdp-reduced}. Then, there exist AKKT sequences $\{x^k\}\subset\R^n$, $\{\lambda^k\}\subset\R^p$, $\{\alpha^k_j\}\subset\R_+$, and $\{\mu^k_j\}\subset\S^{m_j}_+$ such that $x^k\to x^*$ and 
\begin{eqnarray}
\label{akkt-thm-sdp}
\nabla f(x^k)+ J_h(x^k)^T\lambda^k-\sum_{j\in I_N(x^*)} J_{g_j}(x^k)^T\mu_j^k-\sum_{j\in I_R(x^*)} \alpha_j^k\nabla \sigma_{\min}(g_j(x^k))\to0, \label{akkt-stat-sdp} \\
\sigma_i(g_j(x^*))>0\Rightarrow \sigma_i(\mu_j^k)\to 0, \quad  i=1,\ldots,m_j, \quad \forall j\in I_N(x^*),\label{akkt-comp-sdp} 
\end{eqnarray}
where $\sigma_i(\mu_j^k)$ and $\sigma_i(g_j(x^*))$ denote corresponding eigenvalues of $\mu_j^k$ and $g_j(x^*)$, respectively, regarding ordered orthonormal eigenbasis $\{\nu_i(\mu_j^k)\}_{i=1}^{m_j}$ and $\{\nu_i(g_j(x^*))\}_{i=1}^{m_j}$  such that $\nu_i(\mu_j^k)\to \nu_i(g_j(x^*))$ for all $i=1,\ldots,m_j$ and all $j\in I_N(x^*)$.
\end{theorem}

With this result at hand, we proceed in a similar manner to Definition~\ref{def:RCPLD} in order to extend the \emph{Relaxed Constant Positive Linear Dependence} (RCPLD) condition to SDP via problem \eqref{sdp-reduced}.

\begin{definition}
\label{def:RCPLD-sdp}
Let $x^*$ be feasible for problem \eqref{sdp} and let $I\subseteq\{1,\ldots,p\}$ be such that $\{\nabla h_i(x^*)\}_{i\in I}$ is a basis for the space spanned by $\{\nabla h_i(x^*)\}_{i=1}^p$. We say that \emph{Relaxed Constant Positive Linear Dependence}  holds at $x^*$ when, for every $J\subseteq I_R(x^*)$, there exists a neighborhood $V$ of $x^*$ such that:
\begin{itemize}
\item $\{\nabla h_i(x)\}_{i=1}^{p}$ has constant rank for all $x\in V$;
\item If the system
$$
\begin{aligned}
J_h(x^*)^T\lambda + \sum_{j\in I_N(x^*)} J_{g_j}(x^*)^T \mu_j + \sum_{j\in J}  \alpha_j \nabla \sigma_{\min}(g_j(x^*))=0,\\
\lambda\in \R^p, \quad \quad \mu_j\succeq 0, \ \forall j\in I_N(x^*), \quad \quad \alpha_j\geqslant 0, \ \forall j\in J
\end{aligned}
$$
has a nontrivial solution, then $\{\nabla h_i(x)\}_{i\in I}\cup\{\nabla \sigma_{\min}(g_j(x))\}_{j\in J}$ is linearly dependent for every $x\in V$.
\end{itemize}
\end{definition}

Next, we show that RCPLD is a constraint qualification using AKKT sequences (Theorem \ref{thm:akkt-sdp}).

\begin{theorem}\label{sdp:mincpldkkt}
Let $x^*$ be a feasible point of problem \eqref{sdp} satisfying  the AKKT condition~\eqref{akkt-thm-sdp} and RCPLD stated in Definition \ref{def:RCPLD-sdp}. Then, the KKT conditions \eqref{eq:KKT_P} hold at $x^*$. In particular, RCPLD is a constraint qualification. 
\end{theorem}

\begin{proof}
%
%
Let  $\{x^k\}\to x^*$, $\{\lambda^k\}\subset \R^p$, $\{\alpha^k_j\}\subset \R_+$, and $\{\mu^k_j\}\subset \S^{m_j}_+$ be sequences such that \eqref{akkt-stat-sdp} and \eqref{akkt-comp-sdp} hold.
By the constant rank assumption and the definition of $I$, the set $\{\nabla h_i(x^k)\}_{i\in I}$ is a basis for the space spanned by $\{\nabla h_i(x^k)\}_{i=1}^p$ when $k$ is large enough. Hence, for all such $k$, there are new scalars $\tilde{\lambda}^k\in \R^{|I|}$ such that 
$$\sum_{i=1}^p \lambda^k_i \nabla h_i(x^k) = \sum_{i\in I} \tilde\lambda^k_i \nabla h_i(x^k),$$
for all $k$. Set $\tilde\lambda^k_i=0$ for all $i\not\in I$. So, $J_h(x^k)^T \lambda^k=J_h(x^k)^T\tilde \lambda^k$ for all $k$.

Also, thanks to Carath\'eodory's Lemma (Lemma~\ref{lemma:carath}) in \eqref{akkt-stat-sdp}, for every fixed $k$ there is a nonempty subset $J^k\subset I_R(x^*)$ such that $\{\nabla h_i(x^k)\}_{i\in I}\bigcup\{\nabla \sigma_{\min}(g_j(x^k))\}_{j\in J^k}$ is linearly independent and, consequently, \eqref{akkt-stat-sdp} can be rewritten as follows
\begin{equation} \label{akkt-stat-sdp2} 
\nabla f(x^k)+ J_h(x^k)^T\tilde \lambda^k-\sum_{j\in I_N(x^*)} J_{g_j}(x^k)^T\mu_j^k-\sum_{j\in  J^k} \tilde{\alpha}_j^k \nabla \sigma_{\min}(g_j(x^k))\to0,
\end{equation}
for some $\tilde{\alpha}^k_j\geqslant 0$, where $j\in J^k$. Note that in this process the scalars $\tilde\lambda_i^k, i\in I$, also changes, but we abuse the notation by still denoting them by $\tilde\lambda_i^k$. Now, by the infinite pigeonhole principle, we can assume, without loss of generality, that $J^k=J$, for all $k\in \N$. That is, we can take a subsequence if necessary such that $J^k$ does not vary with  $k$. 

 Now, we claim that the sequences $\{\tilde{\lambda}^k\}$,  $\{\mu_j^k\}$, $j\in I_N(x^*)$, and $\{\tilde{\alpha}_j^k\}$, $j\in J$ are bounded. Indeed, set $$M_k:= \max\{\tilde\alpha_j^k, j\in J; \|\mu_j^k\|, j\in I_N(x^*); \|\tilde{\lambda}^k\|\}$$ and suppose that $\{M_k\}$ is unbounded. This implies, by passing to a subsequence if necessary, that 
 \begin{eqnarray*}
  & & -\frac{\tilde{\lambda}_i^k}{M^k}\to\lambda_i\in\R, \: i\in I; \qquad \frac{\mu_j^k}{M^k}\to\mu_j\in K_{m_j}, \: j\in I_N(x^*); \\
  & & \frac{\tilde{\alpha}_j^k}{M^k}\to\alpha_j\geq0, \: j\in J, \qquad \mbox{ with }(\lambda_i,\mu_j,\alpha_j)\neq0.
\end{eqnarray*}
 Then, by dividing \eqref{akkt-stat-sdp} by $M_k$ and passing to the limit, we contradict RCPLD.

Finally, let $\bar{\mu}_j\in \S^{m_j}_+$ ($j\in I_N(x^*)$), $\bar{\alpha}_j\geqslant 0$ ($j\in I_R(x^*)$), and $\bar{\lambda}$, be limit points of the sequences $\{\mu_j^k\}$ ($j\in I_N(x^*)$), $\{\tilde{\alpha}_j^k\}$ ($j\in I_R(x^*)$), and $\{\tilde{\lambda}^k\}$, respectively. Note that these limit points are Lagrange multipliers associated with $x^*$. Indeed, 
by definition of $I_R(x^*)$, we always have $\sigma_{\min}(g_j(x^*))\bar{\alpha}_j=0$, for all $j\in I_R(x^*)$. So, for each $j\in I_R(x^*)$ the matrix $\bar{\mu}_j := \bar{\alpha}_j \nu_{\min}(g_j(x^*))\nu_{\min}(g_j(x^*))^T$ is positive semidefinite and satisfies that $J_{g_j}(x^*)^T\bar{\mu}_j= \bar{\alpha}_j^k \nabla \sigma_{\min}(g_j(x^k))$
(cf. \eqref{eq:der-sigma_min}).
%
Additionally, set $\bar\mu_j:=0$ when $j$ is such that $g_j(x^*) \succ 0$.
Then, it follows from \eqref{akkt-stat-sdp} that 
 $$\nabla f(x^*)+ J_h(x^*)^T\bar \lambda- \sum_{j=1}^\ell J_{g_j}(x^*)^T\bar{\mu}_j=0,$$
 which together with \eqref{akkt-comp-sdp} implies that $g_j(x^*)\cdot \bar{\mu}_j=0$ for every $j$.  
 The desired result follows.
\end{proof}

The CRSC condition can also be extended in a very similar manner. That is, we treat the conic constraints that ``look like equality constraints'' near the feasible point $x^*$, as equality constraints, which means it is not necessary to consider the rank-type structure of every subset of their gradients, but only of one fixed set. To formalize our analyses, we define the set
\begin{equation}\label{sdp:Jminus}
\begin{split}J_{-}(x^*):=\Bigg\{j_0\in I_R(x^*) {\Big |} -\nabla\sigma_{min}(g_{j_0}(x^*))=\sum_{i=1}^p\lambda_i\nabla h_i(x^*) + \hspace{-5pt}\sum_{j\in I_R(x^*)}\hspace{-5pt}\alpha_j\nabla\sigma_{\min}(g_j(x^*)),\\
\mbox{ for some } \lambda_i\in\R, \alpha_j\geq0 \Bigg\},
\end{split}
\end{equation}
and the set $J_+(x^*):=I_R(x^*)\setminus J_-(x^*)$. Now, the \emph{Constant Rank of the Subspace Component} (CRSC) constraint qualification for SDP is defined as follows:

\begin{definition}
Let $x^*$ be a feasible point of \eqref{socp} and $J_{-}(x^*)\subseteq I_R(x^*)$ be defined as in~\eqref{sdp:Jminus}. We also take $I\subseteq\{1,\dots,p\}$ and $J\subseteq J_{-}(x^*)$ such that $\{\nabla h_i(x^*)\}_{i\in I}\cup\{\nabla\sigma_{\min}(g_j(x^*))\}_{j\in J}$ is a basis of the space spanned by the set $\{\nabla h_i(x^*)\}_{i=1}^p\cup\{\nabla\sigma_{\min}(g_j(x^*))\}_{j\in J_{-}(x^*)}$. We say that \emph{Constant Rank of the Subspace Component} (CRSC)  condition holds at $x^*$ when there exists a neighborhood $V$ of $x^*$ such that:
\begin{itemize} 
\item $\{\nabla h_i(x)\}_{i=1}^p\cup\{\nabla\sigma_{\min}(g_j(x))\}_{j\in J_{-}(x^*)}$ has constant rank for all $x$ in $V$;
\item the system
\begin{eqnarray*}\sum_{i\in I}\lambda_i\nabla h_i(x^*)+\sum_{j\in I_N(x^*)}J_{g_j}(x^*)^T\mu_j+\sum_{j\in J\cup J_{+}(x^*)}\alpha_j\nabla\sigma_{\min}(g_j(x^*))=0,\\
\lambda_i\in\R, i\in I; \quad \mu_j\in \S^{m_j}_+, j\in I_N(x^*); \quad \alpha_j\in\R, j\in J; \quad \alpha_j\geq0, j\in J_{+}(x^*),
\end{eqnarray*}
has only the trivial solution.
\end{itemize}
\end{definition}

It is possible to prove that CRSC is indeed a constraint qualification, but since the proof follows from the same arguments  provided in the proof of Theorem~\ref{sdp:mincpldkkt}, it is omitted. The next counterexample,  analogous to Example~\ref{ex:socp}, shows that CRSC and RCPLD are strictly weaker than Robinson's CQ. 

\begin{example}\label{ex:sdp}
Consider the following pair of constraints: 
$$
g_1(x):= \half
\begin{bmatrix}
x+1 & x-1\\
x-1 & x+1
\end{bmatrix}
\in \S^2_+,
\quad
g_2(x):= \half \begin{bmatrix}
1-x & -x-1\\
-x-1 & 1-x
\end{bmatrix}
\in \S^2_+
$$
and the point $x^* = 0$, which is the unique feasible point. The eigenvalues of $g_1(x)$ are $\sigma_{\min}(g_1(x))=x$ and $\sigma_{\max}(g_1(x))=1$, with corresponding eigenvectors $\nu_{\min}(g_1(x))=(1,1)^T$ and $\nu_{\max}(g_1(x))=(1,-1)^T$, respectively, for all $x$ close to $x^*$. With the same eigenvectors, the eigenvalues of $g_2(x)$ are $\sigma_{\min}(g_2(x))=-x$ and $\sigma_{\max}(g_2(x))=1$, when $x$ is close to $x^*$. 

Also, note that $\sigma_{\min}(g_1(x^*))$ and $\sigma_{\min}(g_2(x^*))$ are both simple, which means the reformulation of the problem as in~\eqref{sdp-reduced} is simply an NLP problem. Moreover, we have that $\nabla \sigma_{\min}(g_1(x))=1$, $\nabla  \sigma_{\min}(g_2(x))=-1$, for all $x$ close enough to $x^*=0$. Then, RCPLD and CRSC (with $J_-(x^*)=\{1,2\}$ and, consequently, $J_+(x^*)=\emptyset$ and $J$ equals either $\{1\}$ or $\{2\}$) hold. However, Robinson's CQ does not hold. Thus, RCPLD and CRSC are strictly implied by Robinson's CQ.
\end{example}

\section{Conclusion} 

We have presented naive definitions of constant rank-type CQs for second-order cone programming and semidefinite programming. The definition is naive in the sense that no improvement is made with respect to irreducible constraints, where our definitions resume to Robinson's CQ. However, in general, our definitions are strictly weaker than Robinson's CQ. In order to present a definition that takes into account the true conic constraints, we expect that a much more involving implicit function approach or Approximate-KKT approach would be needed, which is a subject of current research. Note that, since augmented Lagrangian algorithms described in \cite{psocp} and \cite{AHV} generate an AKKT sequence for SOCP \eqref{socp} and  SDP  \eqref{sdp} problems, respectively,  CQs introduced in these notes are sufficient for showing global convergence to a KKT point without assuming Robinson's CQ.\\


\si{\section*{Acknowledgement} \FINANCIAMENTO}

\jo{\section*{Conflict of interest}
The authors declare that they have no conflict of interest.}

%
%

\si{\bibliographystyle{plain}}
\jo{\bibliographystyle{spmpsci}}
\bibliography{biblio}

\end{document}